\documentclass[12pt, reqno]{amsart}
\usepackage{amsaddr}
\usepackage{graphicx, amssymb, amsthm, amsfonts, latexsym, epsfig, amsmath, amscd, amsbsy}
\usepackage[usenames]{color}
\usepackage{tikz, caption, subcaption}
\usepackage{enumerate}
\usepackage{epstopdf}

\usetikzlibrary{decorations.pathmorphing}
\usetikzlibrary{calc}

\usetikzlibrary{arrows, patterns, calc}
\pgfdeclarelayer{bg}
\pgfsetlayers{bg,main}

\pgfdeclarelayer{bg}
\pgfsetlayers{bg,main}

\numberwithin{equation}{section}
\theoremstyle{plain}
\newtheorem{theorem}{Theorem}[section]
\newtheorem{lem}[theorem]{Lemma}
\newtheorem{cor}[theorem]{Corollary}

\theoremstyle{definition}
\newtheorem{df}[theorem]{Definition}
\newtheorem{rem}[theorem]{Remark}

\renewcommand{\phi}{\varphi}

\def\Int{\mathop\mathrm{Int}}
\def\Cl{\mathop\mathrm{Cl}}
\def\dist{\mathop\mathrm{dist}}

\def\Conv{\mathop\mathrm{Conv}}

\newcommand\restrict[1]{\raisebox{-.5ex}{$|$}_{#1}}

\usepackage{hyperref}
\hypersetup{
    colorlinks=true,
    linkcolor=blue,
    citecolor=violet,
    pdftitle={Accumulation Sets and Zero Entropy Dynamics in the Lozi Map},
    pdfauthor={Kristijan Kilassa Kvaternik}
    }

\begin{document}
	
	\title{Accumulation Sets and Zero Entropy Dynamics in the Lozi Map}
	\author[Kristijan Kilassa Kvaternik]{Kristijan Kilassa Kvaternik}
	\address[K. Kilassa Kvaternik]{
		University of Zagreb,
		Faculty of Electrical Engineering and Computing,
		Department of Applied Mathematics,
		Unska 3,
		10\,000 Zagreb,
		Croatia -- and -- Jagiellonian University,
		Institute of Mathematics,
		ul.\ prof.\ Stanis{\l}awa {\L}ojasiewicza 6,
		30-348 Kraków,
		Poland}
	%\urladdr{\url{https://www.fer.unizg.hr/en/kristijan.kilassa_kvaternik}}
	\email{kristijan.kilassakvaternik@fer.unizg.hr}
		
	\date{April 24, 2026}
	
	
	\begin{abstract}
				
		For the Lozi map $L_{a,b}$, we consider parameter pairs for which the fixed point $X$ has no homoclinic points and the period-two orbit $\{P,P'\}$ is attracting. For such parameters, let $\ell$ be the set of accumulation points of the unstable manifold $W_X^u$ that do not lie on $W_X^u$. We construct a polygon $\mathcal{D}$ whose forward images under $L_{a,b}$ form nested sequences of sets that eventually become trapping. We show that this geometric construction gives a characterization of $\ell$ as the intersection of these iterates. Using this structure, we prove that the non-wandering set for $L_{a,b}^2$ is contained in the union of $\ell$ and the set of fixed points of $L_{a,b}$. As a consequence, the Lozi map, restricted to the complement of $\ell$ in the plane, has zero topological entropy. This result extends a recent one of Misiurewicz and \v{S}timac to a broader set of parameters. 
		
	\end{abstract}
	
	\maketitle	
	
	\baselineskip=18pt
	
	\textbf{In the theory of dynamical systems, understanding the complexity of a given system is of fundamental importance. One of the principal tools for assessing such complexity is topological entropy, a quantity that measures the growth rate of distinguishable states over time up to a prescribed level of precision. In particular, systems with positive topological entropy exhibit substantially richer and more intricate dynamics than those with zero entropy, for which the number of distinguishable states remains essentially unchanged. Consequently, when a system depends on parameters, it becomes important to determine the parameter values at which the entropy transitions from zero to positive. A central role in this analysis is played by the non-wandering set, a set consisting of states that return arbitrarily close to their initial positions as the dynamics evolve. Such states do not "drift away," but instead display recurrent behavior and capture the essential information governing the complexity of the system. In this paper, we investigate the Lozi map family. The Lozi family is a family of piecewise affine transformations of the plane exhibiting chaotic behavior. Beyond their relevance in practical applications, such as cryptography and memristor modeling, they also serve as an important model for the study of other piecewise affine systems as well as the Hénon family, one of the most extensively studied classes of dynamical systems. For a broad range of parameters, we determine the non-wandering set of the Lozi map and show that it is closely connected to a distinguished accumulation set. We provide a geometric description of this accumulation set by constructing a specific polygon that traps all states contained within it. Moreover, we show that, under iteration, this polygon increasingly approximates the accumulation set. As a consequence, we prove that the topological entropy of the Lozi map outside the accumulation set is equal to zero. In other words, positive entropy and dynamically complex behavior are concentrated on a geometrically defined accumulation set. Our results extend the earlier work of Micha\l{} Misiurewicz and Sonja \v{S}timac to a broader class of parameters and represent a further step toward identifying the boundary between regions of zero entropy and the onset of more intricate, chaotic dynamics in the Lozi family.}
	
	\baselineskip=18pt
	
	\section{Introduction}\label{sec:intro}

	The Lozi map is a piecewise affine counterpart to the H\'enon map: it is a planar homeomorphism given by
	\begin{equation*}
		L_{a,b} \colon \mathbb{R}^2 \rightarrow \mathbb{R}^2, \quad L_{a,b}(x,y) = (1 + y - a|x|, bx),
	\end{equation*}
and it depends on two parameters $a,b \in \mathbb{R}$, $b \neq 0$. Almost half a century after R.\ Lozi proposed this map in \cite{lozi1978attracteur}, a complete understanding of its dynamical properties still requires substantial research. One such aspect is topological entropy $h_{top}$, a core dynamical invariant that measures the exponential growth of orbits for a given system.

	A general formula for the topological entropy of the Lozi map in terms of the parameters is not known, but there are some partial results on its discontinuity and jumps (see \cite{yildiz2012discontinuity}), as well as its monotonicity along specific directions in the parameter space (see \cite{ishii1997towards2}, \cite{ishii1998monotonicity}, \cite{yildiz2011monotonicity}). Although the parameter region where $h_{top}(L_{a,b})$ is maximal has been completely determined in \cite{ishii1997towards1}, the zero entropy locus of the Lozi family is still not fully understood.
	
	In \cite{yildiz2011monotonicity}, Yildiz reviewed a result of Ishii and Sands which states that $h_{top}(L_{a,b})=0$ when $-1 \leqslant b < 0$ and $a \leqslant b - 1$, as well as when $0 < b \leqslant 1$ and $a \leqslant 1 - b$. In addition, Yildiz also shows that entropy is zero for parameters in a small neighborhood of the point $(a,b) = (1,0.5)$.
	
	In a recent paper \cite{misiurewicz2024zero}, Misiurewicz and \v{S}timac extended this last result to a larger parameter set in which there are no homoclinic points for the fixed point $X$ in the first quadrant, and the period-two orbit $\{P,P'\}$ is attracting. We denote that set $\mathfrak{R}$ and formally introduce it at the beginning of Section \ref{sec:non_wandering_zero_entropy}. Within that set, Misiurewicz and \v{S}timac observed a subset $\mathcal{R} \subset \mathfrak{R}$, in which they showed that the branches of the unstable manifold of $X$, $W_X^u$, are attracted to the set $\ell = \{P,P'\}$ because $L_{a,b}$ acts globally on these branches as an affine map. This allowed them to conclude that $h_{top}(L_{a,b}) = 0$ for all $(a,b) \in \mathcal{R}$.
	
	In this paper, we generalize this result to the set $\mathfrak{R} \setminus \mathcal{R}$, where the aforementioned dynamical behavior is disrupted. In particular, this gives rise to a more complex accumulation set $\ell$, which is also suggested by certain counterexamples that Ishii and Sands pointed out to the author in a private communication; see Section \ref{sec:concluding_remarks}. More precisely, we consider the set $\ell := \Cl W_X^u \setminus W_X^u$ of all accumulation points of $W_X^u$ which do not lie on $W_X^u$. We construct a geometric mechanism that generates $\ell$. First, we define a polygon $\mathcal{D}$ and show that $\mathcal{D}$ is an eventually trapping region for $L_{a,b}^2$ (in the sense of Definition \ref{dfn:eventually_trapping_region}). We then consider the nested sequences $(L_{a,b}^{2n}(\mathcal{D}))_{n \in \mathbb{N}_0}$ and $(L_{a,b}^{2n+1}(\mathcal{D}))_{n \in \mathbb{N}_0}$. As it turns out, the connected components of $\ell$ can be represented as intersections of these two sequences (Lemma \ref{lem:P_intersection_D'}, Corollary \ref{cor:ell_intersect_D}). This representation ultimately permits us to localize the non-wandering dynamics for $L_{a,b}^2$:
	
	\begin{theorem} \label{thm:L2_non_wandering_set}
	For all $(a,b)\in\mathfrak{R}\setminus\mathcal{R}$, the non-wandering set of $L_{a,b}^2$ is contained in the union of $\ell$ and the fixed points of $L_{a,b}$, that is,
	\begin{equation*}
		\Omega(L_{a,b}^2) \subseteq \ell \cup \{X,Y\}.
	\end{equation*}
	\end{theorem}
	
Such structure of the non-wandering set gives an immediate consequence for the topological entropy of $L_{a,b}$: we isolate all entropy-producing dynamics inside a geometrically defined limit set $\ell$. 

	\begin{cor} \label{cor:entropy_gothR}
	For all parameters $(a,b)\in\mathfrak{R}\setminus\mathcal{R}$, we have $h_{top}(L_{a,b}\restrict{\mathbb{R}^2\setminus\ell})=0$.
	\end{cor}

	As already stated, there are examples of parameters $(a,b) \in \mathfrak{R} \setminus \mathcal{R}$ indicating that the Lozi map can have positive entropy in that parameter region. Therefore, $\ell$ can, in general, have a nontrivial structure and contain more points than the periodic orbit $\{P,P'\}$. This intricate structure of $\ell$ depends on the specific dynamical behavior of $L_{a,b}$ and the unstable manifold $W_X^u$ and thus merits its own separate study. However, in this paper, we construct a geometric mechanism that generates $\ell$ and localizes the dynamical complexity of $L_{a,b}$ to this set. Moreover, notice that this construction of an eventually trapping region suggests that $\ell$ behaves like an attractor. Still, we cannot draw an analogy with the strange attractors from the Misiurewicz parameter set introduced in \cite{misiurewicz1980strange}: in our case, $\ell$ is the $\omega$-limit set of $W_X^u$ but does not contain it, since, by Corollary \ref{cor:eps}, every compact subarc of $W_X^u$ has a neighborhood in which it is the only part of $W_X^u$. Hence, this new, invariant set $\ell$ motivates a potential new study of another class of attractors that could arise in the Lozi family.
	
	This paper is organized as follows: in Section \ref{sec:prelim}, we give an overview of relevant notions and results from topological dynamics and the dynamics of the Lozi map, together with the notation. Section \ref{sec:non_wandering_zero_entropy} consists of two parts. In Subsection \ref{subsec:accumulation_set_ell}, we construct an eventually trapping region $\mathcal{D}$ and use it to characterize the accumulation set $\ell$. We use this characterization to prove our main theorem and its consequence in Subsection \ref{subsec:main_result}. Finally, we provide some concluding remarks in Section \ref{sec:concluding_remarks}.
	
	This paper is based on Chapter 3 of the author's PhD dissertation \cite{kilassa2022tangential}, where preliminary versions of these results first appeared.

	\section{Preliminaries}\label{sec:prelim}
	
	\subsection{Topological dynamics}\label{subsec:dynsys_general_notions}
	
	Let $\mathcal{X}$ be a topological space and $f \colon \mathcal{X} \rightarrow \mathcal{X}$ a continuous function. For $n \in \mathbb{N}$, let $f^n = f \circ f \circ \ldots \circ f$ ($n$ times). In particular, $f^0$ is the identity map. In addition, if $f$ is invertible, let $f^{-n} = f^{-1} \circ f^{-1} \circ \ldots \circ f^{-1}$ ($n$ times). A set $\mathcal{Y} \subseteq \mathcal{X}$ is said to be $f$-\emph{invariant} if $f(\mathcal{Y}) \subseteq \mathcal{Y}$.
	
	For a point $x \in \mathcal{X}$, the \emph{forward orbit} of $x$ is the set $\mathcal{O}^{+}(x,f)=\{f^n(x) \colon n \in \mathbb{N}_0\}$. If $f$ is invertible, the \emph{backward orbit} of $x$ is $\mathcal{O}^{-}(x,f) = \{f^{-n}(x) \colon n \in \mathbb{N}_0\}$, while the \emph{full orbit} of $x$ is $\mathcal{O}(x,f) = \mathcal{O}^{+}(x,f) \cup \mathcal{O}^{-}(x,f)$.
	
	We say that a point $x \in \mathcal{X}$ is a \emph{non-wandering point} if for every neighborhood $\mathcal{U}$ of $x$, there exists $n \in \mathbb{N}$ such that $f^{n}(\mathcal{U})\cap \mathcal{U} \neq \emptyset$. The set of all non-wandering points will be denoted by $\Omega(f)$. If $f$ is a homeomorphism, we know that $\Omega(f)$ is both $f$-invariant and $f^{-1}$-invariant. A point that does not lie in $\Omega(f)$ is called a \emph{wandering point} for $f$.
	
	Furthermore, for a subset $\mathcal{A} \subset \mathcal{X}$, the $\omega$-\emph{limit set} of $\mathcal{A}$ is defined by
	\begin{equation*}
	\begin{split}
	\omega(\mathcal{A},f) & = \bigcap_{n \in \mathbb{N}} \Cl\bigcup_{k \geqslant n}\{f^k(a) \colon a \in \mathcal{A}\} \\ 
				& = \{ x \in \mathcal{X} \colon \exists \text{ a sequence } (n_k)_k\subseteq\mathbb{N}_0 \text{ such that } n_k \rightarrow \infty \\
				& \qquad\qquad\text{ and a sequence } (a_k)_k \subseteq \mathcal{A} \text{ such that } f^{n_k}(a_k) \xrightarrow{k \rightarrow \infty} x \},
	\end{split}
	\end{equation*}
	where $\Cl \mathcal{Y}$ is the topological closure of a set $\mathcal{Y} \subseteq \mathcal{X}$. The set $\omega(\mathcal{A},f)$ is closed, $f$-invariant and $f^{-1}$-invariant if $f$ is a homeomorphism. 
	
	 Let $\mathcal{X}$ now be a compact metric space with a metric $d$. Moreover, let $n$ be a positive integer and $\varepsilon>0$. A finite set $\mathcal{S}\subseteq \mathcal{X}$ is called an $(n,\varepsilon)$-\emph{separated set} if for any two distinct points $x,y \in \mathcal{S}$, there exists $i \in \{0,1,\ldots,n-1\}$ such that $d(f^{i}(x),f^{i}(y)) > \varepsilon$. Let $s(n,\varepsilon)$ denote the maximal cardinality of an $(n,\varepsilon)$-separated set. The \emph{topological entropy} $h_{top}(f)$ of $f$ is defined as
	\begin{equation*}
	 h_{top}(f) = \lim_{\varepsilon \rightarrow 0^{+}} \limsup_{n \rightarrow \infty} \frac{1}{n}\log(s(n,\varepsilon)).
	\end{equation*}	    
	Here, $s(n,\varepsilon)$ can be interpreted as the number of orbit segments of length $n$ that one can detect up to precision $\varepsilon$. Therefore, $h_{top}(f)$ estimates the average exponential growth of the number of detectable orbit segments. It is also well-known that for every $n \in \mathbb{N}$, we have $h_{top}(f^n)=nh_{top}(f)$ (see, e.g., Proposition 2.5.5 on page 39 in \cite{brin2002introduction}). In addition, by Equation (3.3.1) on page 130 in \cite{katok1995introduction}, we have that 
	\begin{equation} \label{eq:entropy_non_wandering_restriction}
	h_{top}(f)=h_{top}(f \restrict{\Omega(f)}).
	\end{equation}
	
	\begin{rem}
	In this paper, we are interested in the topological entropy of the Lozi map $L_{a,b}$. This map is defined on $\mathbb{R}^2$, which is not a compact set. Therefore, as in \cite{yildiz2012discontinuity}, we take the one-point compactification of $\mathbb{R}^2$ and extend $L_{a,b}$ continuously to that set. 
	\end{rem} 
	
	Finally, we state a classic result for homeomorphisms of the plane as it was formulated (and proved) in \cite{franks1992new}. For a homeomorphism $f \colon \mathbb{R}^2 \rightarrow \mathbb{R}^2$, a \emph{domain of translation} for $f$ is an open, connected subset of $\mathbb{R}^2$ whose boundary is $p \cup f(p)$, where $p$ is a properly embedded line in $\mathbb{R}^2$, and $p$ separates $f^{-1}(p)$ and $f(p)$.
	
	\begin{theorem}[Brouwer Plane Translation Theorem, BPTT]
	Let $f \colon \mathbb{R}^2 \rightarrow \mathbb{R}^2$ be a fixed-point-free, orientation-preserving homeomorphism of the plane. Then every point of the plane is contained in a domain of translation for $f$. 
	\end{theorem}
	
	In this paper, we will use the following consequence of BPTT and Corollary 1.3 in \cite{franks1992new}; the proof is omitted.
	
	\begin{cor} \label{cor:BPTT_corollary}
	If $f \colon \mathbb{R}^2 \rightarrow \mathbb{R}^2$ is a fixed-point-free, orientation-preserving homeomorphism of the plane, then every point in the plane is a wandering point for $f$.
	\end{cor}

	\subsection{Lozi map}\label{subsec:lozi_maps}
	
	We recall here some definitions and facts already explained in \cite{kilassa2024tangential}. As already stated there, to study the dynamics of the orientation-reversing Lozi map $L_{a,b}$, one can observe parameter pairs $(a,b)$ such that $0 < b < 1$ and $a + b > 1$.

	For all such parameter pairs, $L_{a,b}$ has two hyperbolic saddle fixed points, $X=\bigl(\frac{1}{1+a-b},\:\frac{b}{1+a-b}\bigr)$ in the first and $Y=\bigl(\frac{1}{1-a-b},\:\frac{b}{1-a-b}\bigr)$ in the third quadrant. The eigenvalues of the differential of $L_{a,b}$ at $X$ are
	\begin{equation*}
	\lambda_X^u = \tfrac{1}{2}\left(-a-\sqrt{a^2+4b}\right),\quad \lambda_X^s = \tfrac{1}{2}\left(-a+\sqrt{a^2+4b}\right),
	\end{equation*}
Observe that $\lambda_X^u<-1$, $0<\lambda_X^s<1$. Moreover, for every eigenvalue $\lambda$, the corresponding eigenvector is given by $\binom{\lambda}{b}$.

	In addition, when $1-b<a<1+b$, there are also two attracting periodic points of period two,
	\begin{equation*}
	P=\left(\frac{1+a-b}{a^2+(1-b)^2},\,\frac{b(1-a-b)}{a^2+(1-b)^2}\right),\ P'=\left(\frac{1-a-b}{a^2+(1-b)^2},\,\frac{b(1+a-b)}{a^2+(1-b)^2}\right).
	\end{equation*}
Points $P$ and $P'$ lie in the fourth and second quadrant, respectively. 

	For a point $A \in \mathbb{R}^2$ and $k \in \mathbb{Z} \setminus \{0\}$, let $A^k=L_{a,b}^k(A)$; in particular, $A^0=A$.

	Recall that the \emph{unstable manifold} of $X$ is the set $W_X^u = \{T \in \mathbb{R}^2 \colon T^{-n} \overset{n \rightarrow \infty}{\longrightarrow} X\}$, while the \emph{stable manifold} of $X$ is $W_X^s = \{T \in \mathbb{R}^2 \colon T^n \overset{n \rightarrow \infty}{\longrightarrow} X\}$. It is known that $W_X^u$ and $W_X^s$ are $L_{a,b}$-invariant and $L_{a,b}^{-1}$-invariant sets, and they contain $X$; see Figure \ref{fig:Lozi_stable_unstable_Z_V}. Here, we follow the standard convention and call these sets manifolds, as they were introduced in \cite{misiurewicz1980strange}, even though these sets are polygonal lines and thus not manifolds in the true sense of that notion.

	\begin{figure}
	\begin{center}
		\includegraphics[width=\linewidth]{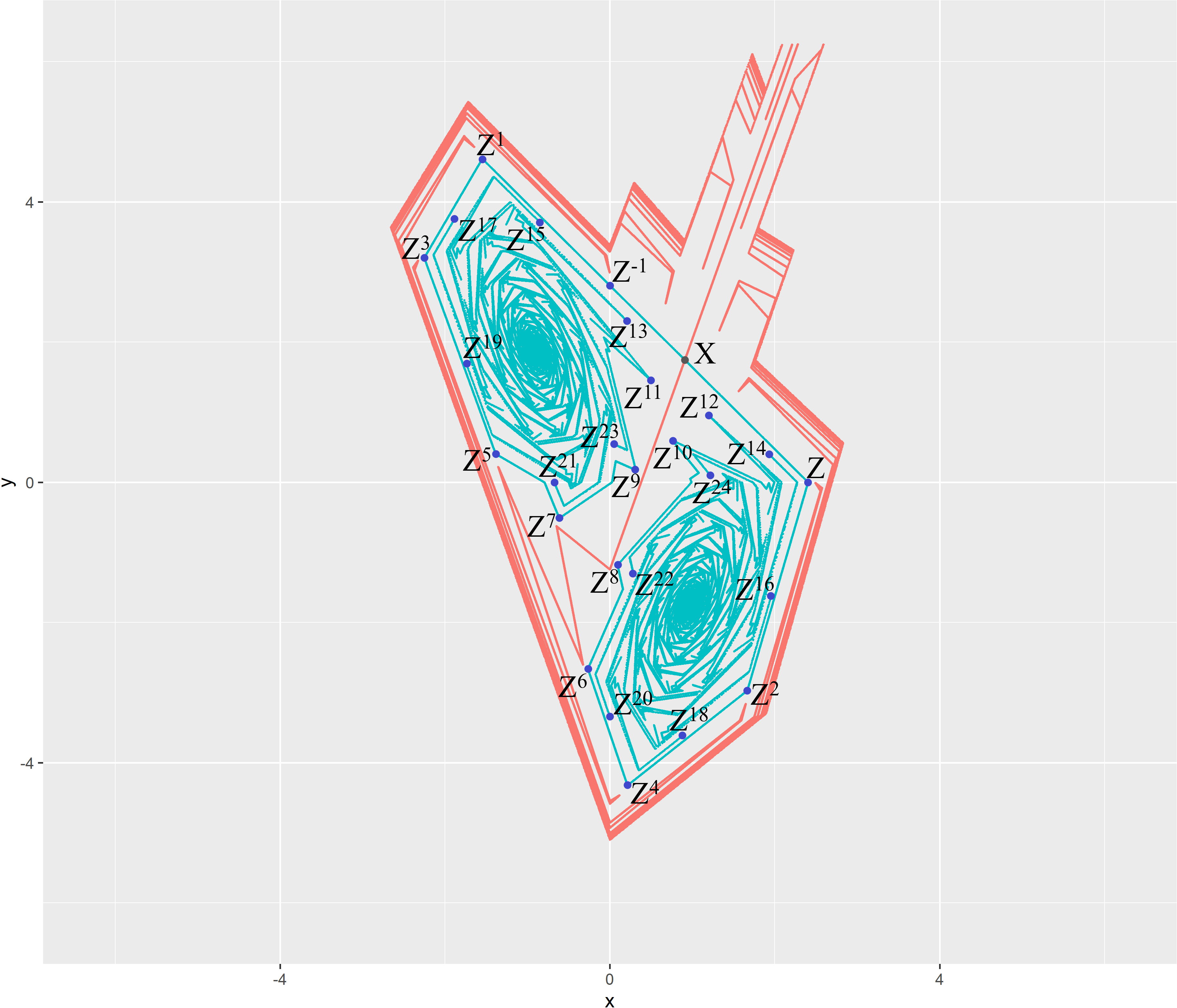}
	\end{center}
	\caption{The stable (red) and unstable (blue) manifold of $X$ for parameter values $a=1.06$, $b=0.96$, together with some iterates of $Z$.\label{fig:Lozi_stable_unstable_Z_V}}
	\end{figure}
	
	Observe the unstable manifold $W_X^u$. We denote the half of that manifold that starts at $X$ and goes to the right by $W_X^{u+}$. That half intersects the positive $x$-axis for the first time at the point
	\begin{equation}
	Z=\left(\frac{2+a+\sqrt{a^2+4b}}{2(1+a-b)},\: 0\right)=\left(\frac{2}{2+a-\sqrt{a^2+4b}},\: 0\right).
	\label{eq:Z}
	\end{equation}
We denote the other half, starting at $X$ and going to the left, by $W_X^{u-}$. Note that
	\begin{equation*}
	W_X^{u+} = \{X\}\cup\bigcup_{n=-\infty}^{\infty}L_{a,b}^{2n}\bigl(\overline{ZZ^2}\bigr) ,\quad W_X^{u-} = \{X\}\cup\bigcup_{n=-\infty}^{\infty}L_{a,b}^{2n}\bigl(\overline{Z^{-1}Z^1}\bigr).
	\end{equation*}
Similarly, we denote by $W_X^{s-}$ the half of the stable manifold $W_X^s$ that starts at $X$ goes down. The other half of $W_X^s$ is a half-line emanating from $X$ and going up in the first quadrant. We denote that part by $W_X^{s+}$.
	
	\begin{lem} \label{lem:stable_iterates_1st_quadrant}
	Let $A$ be a point in the first quadrant below the segment $\overline{ZZ^{-1}}$ and such that $A \notin W_X^s$. Then there exists $m \in \mathbb{N}$ such that $A^{i}$ lies in the first quadrant for all $i \in \{0,1,\ldots,m-1\}$, and $A^{m}$ lies in the second quadrant.
	\end{lem}
	
	\begin{proof}
Assume, by contradiction, that all forward iterates of $A$ remain in the first quadrant below $\overline{ZZ^{-1}}$, that is, in the triangle $OZZ^{-1}$. On this region, $L_{a,b}$ is affine and has a saddle fixed point $X$. Since $A\notin W_X^s$, its orbit has a nonzero component in the unstable direction. Consequently, while the orbit remains in this region, its distance from $W_X^s$ grows by a factor of $|\lambda_X^u|>1$ at each iteration. Hence the orbit eventually leaves the triangle $OZZ^{-1}$, which is a contradiction. Therefore, there exists $m\in\mathbb{N}$ such that $A^i$ lies in the first quadrant for $i=0,\ldots,m-1$, while $A^m$ lies in the second quadrant.
	\end{proof}
	
	\begin{cor} \label{cor:stable_iterates_1st_quadrant}
	If a point $A$ is such that the forward orbit $\mathcal{O}^{+}(A,L_{a,b})$ is contained in the first quadrant below the line $\overline{ZZ^{-1}}$, then $A \in W_X^s$.
	\end{cor}

	\subsection{Notation}\label{subsec:notation}
	
	We will recall and slightly adjust the notation from \cite{kilassa2024tangential}, which will be used in this paper as well.
	
	The first, second, third and fourth quadrant of the Cartesian coordinate system will be denoted by $\mathcal{Q}_1$, $\mathcal{Q}_2$, $\mathcal{Q}_3$ and $\mathcal{Q}_4$, respectively. We will mark the $x$-axis and $y$-axis by $Ox$ and $Oy$, respectively. We use capital letters such as $A,B,C,\ldots$ to denote points in the plane whenever no other notation is specified. The notation $L_{a,b}$ is reserved for the Lozi map. In particular, $X$ and $Y$ are the fixed points of $L_{a,b}$, while $Z$ is the point on $W_X^u$ given by (\ref{eq:Z}); see Figure \ref{fig:Lozi_stable_unstable_Z_V}. For $A,B\in\mathbb{R}^2$, $\overline{AB}$ represents the straight line segment with endpoints $A$ and $B$. For $A\in\mathbb{R}^2$ and $\varepsilon>0$, $B_{\varepsilon}(A)$ is the open ball in the plane centered at $A$ of radius $\varepsilon$.
	
	Two-dimensional subsets of the plane are marked using capital script letters: $\mathcal{A},\mathcal{B}, \ldots$ Let $\dist$ denote the Euclidean metric on $\mathbb{R}^2$. For $\mathcal{A} \subset \mathbb{R}^2$, let $\Int\mathcal{A}$ be the interior of $\mathcal{A}$, $\Cl\mathcal{A}$ the closure of $\mathcal{A}$, $\partial\mathcal{A}$ the boundary of $\mathcal{A}$, and $\Conv\mathcal{A}$ the convex hull of $\mathcal{A}$.		
	
	In addition, we also employ specific notation concerning the Lozi map and the unstable manifold $W_X^u$. For $k\in\mathbb{Z}$ and a point $A \in \mathbb{R}^2$, let $A^k=L_{a,b}^{k}(A)$. In particular, $A^{0} = A$. For points $A,B\in W_X^u$, we define:
			\begin{itemize}
				\item $[A,B]^{u}\subset W_X^u$ is the polygonal line lying on $W_X^{u}$ with $A$ and $B$ as endpoints,
				\item specifically, if $[A,B]^{u}$ is a straight line segment, we denote it by $\overline{AB}^{u}$,
				\item $[A,B)^{u}:=[A,B]^{u}\setminus\{B\}$,
				\item $(A,B]^{u}:=[A,B]^{u}\setminus\{A\}$,
				\item $(A,B)^{u}:=[A,B]^{u}\setminus\{A,B\}$.
			\end{itemize}
Finally, $\ell$ will represent the set of accumulation points of $W_X^u$ not belonging to $W_X^u$ itself, that is, $\ell = \Cl W_X^u \setminus W_X^u$.

	\section{Non-wandering set and entropy outside $\ell$}\label{sec:non_wandering_zero_entropy}
	
	In this section, we want to give additional insight into the non-wandering set and topological entropy of the Lozi map while strengthening the results proven in \cite{yildiz2011monotonicity} and \cite{misiurewicz2024zero}. We prove that the non-wandering set $\Omega(L_{a,b}^2)$ is contained in the union of the set $\ell$ and the set of the fixed points $\{X,Y\}$ of $L_{a,b}$. As a consequence, the topological entropy of $L_{a,b}$ is zero outside the accumulation set $\ell$ of $W_X^u$. 
	
	We consider parameter pairs $(a,b)$, such that $0<b<1$, $1-b<a<1+b$ (recall that the period-two orbit $\{P,P'\}$ is attracting in that case), and such that there are no homoclinic points for the fixed point $X$, that is, $W_X^s \cap W_X^u = \{X\}$. We will denote that region in the parameter space by $\mathfrak{R}$.
	
	In \cite{misiurewicz2024zero}, the authors considered a region $\mathcal{R}$ contained in $\mathfrak{R}$, for which $W_X^u$ intersects the coordinate axes at $Z^{-1}$ and $Z$ only (for every $k\in\mathbb{N}_0$, point $Z^{2k}$ lies in the right, and $Z^{2k-1}$ in the left half-plane). The authors proved that $h_{top}(L_{a,b})=0$ for all $(a,b)\in \mathcal{R}$. Using the same proof technique, this result can be generalized to the case when $W_X^u$ intersects the coordinate axes at finitely many points. For completeness and to recall that technique, we will present a full proof of this claim.
	
	Recall that a topological space $\mathcal{X}$ is said to be \emph{simply connected} if every loop in $\mathcal{X}$ can be continuously contracted to a point.
	
	\begin{theorem}\label{thm:zero_entropy_axes_finitely_many}
	Let $(a,b)\in\mathfrak{R}$ and assume that $W_X^u$ intersects the coordinate axes at finitely many points. Then $h_{top}(L_{a,b})=0$.
	\end{theorem}
	
	\begin{proof}
	If $W_X^u$ intersects the coordinate axes at finitely many points, then there exists a point $A\in W_X^{u+}$ such that $[A,A^1]^u$ contains all intersections of $W_X^u$ with the coordinate axes. Let us consider the sets $ZR:=W_X^{u+}\setminus [X,A)^u$, $ZL:=W_X^{u-}\setminus [X,A^1)^u$. Notice that $ZR$ is contained in the fourth quadrant $\mathcal{Q}_4$, and $ZL$ is contained in $\mathcal{Q}_2$. In addition, $L_{a,b}(\Conv(ZR))=\Conv(ZL)$, $L_{a,b}(\Conv(ZL))\subset\Conv(ZR)$, and the restrictions $L^2_{a,b}\restrict{\Conv(ZR)}$, $L^2_{a,b}\restrict{\Conv(ZL)}$ are both affine maps. Therefore, the union of $\Conv(ZR)$ and $\Conv(ZL)$ is attracted to the periodic orbit $\{P,P'\}$ under iterations of $L_{a,b}^2$. Consequently, $W_X^u\cup\{P,P'\}$ is connected, compact, $L_{a,b}$-invariant and $L_{a,b}^{-1}$-invariant.
	
	Recall that $W_X^{s+}\subset W_X^s$ is a half-line starting at $X$ and going up in $\mathcal{Q}_1$, and denote by $W_Y^{u+}$ the connected component of the unstable manifold of the other fixed point $Y$ which is a half-line starting from $Y$ and going down in $\mathcal{Q}_3$. In the extended plane, let $\mathcal{M}$ be the union
	\begin{equation*}
	\mathcal{M}:=W_X^u\cup\{P,P'\}\cup W_X^{s+}\cup W_Y^{u+}\cup\{\infty\}.
	\end{equation*}
Moreover, let $U$ denote the complement of $\mathcal{M}$ in the plane, $U=\mathbb{R}^2\setminus\mathcal{M}$. The set $\mathcal{M}$ is both $L_{a,b}$-invariant and $L_{a,b}^{-1}$-invariant, compact, connected in the extended plane, and it does not separate either the extended plane or the plane. Hence, $U$ is open and simply connected in the plane, and therefore homeomorphic to the unit disc, and, by extension, to the whole plane.

	In addition, since $\mathcal{M}$ is both $L_{a,b}$-invariant and $L_{a,b}^{-1}$-invariant, then the same holds for its complement $U$. In particular, $U$ does not contain any fixed points of $L_{a,b}^2$.
	
	Finally, notice that the restriction $L_{a,b}^2\restrict{U}$ is an orientation preserving, fixed-point-free homeomorphism. Therefore, by Corollary \ref{cor:BPTT_corollary}, every point of $U$ is a wandering point for $L_{a,b}^2$, and the non-wandering set of $L_{a,b}^2$ consists of its fixed points only. It follows that $0 = h_{top}(L_{a,b}^2) = 2h_{top}(L_{a,b})$, which completes the proof.    
	\end{proof}
	
	Therefore, from now on, we can only observe parameter pairs $(a,b)\in\mathfrak{R}$ for which $W_X^u$ intersects the coordinate axes at infinitely many points.		  	

	\subsection{Accumulation set $\ell$} \label{subsec:accumulation_set_ell}
	
	In this subsection, we consider the set $\ell = \Cl W_X^u \setminus W_X^u$. We show that $\ell$ can be obtained as the intersection of the forward images of a suitably constructed polygon $\mathcal{D}$, which will serve as an eventually trapping region, that is, it is mapped to its interior after several iterations of $L_{a,b}$.
	
	Theorem \ref{thm:zero_entropy_axes_finitely_many} shows that the case of finitely many intersections of $W_X^u$ with the coordinate axes leads to zero entropy. Therefore, in the remainder of the section, we restrict our attention to the case where $W_X^u$ intersects the coordinate axes infinitely many times. In this setting, a new geometric construction is required. The following lemma establishes the existence of a point $S \in W_X^u \cap Ox$ such that the segment $\overline{SZ}$ does not contain further intersections with $W_X^u$. This point serves as the starting point for the construction of the eventually trapping region introduced below.
	
	\begin{lem}\label{lem:point_S}
	Let $(a,b) \in \mathfrak{R}$, and assume that $W_X^u$ intersects the coordinate axes at infinitely many points. Then there exists a point $S\in W_X^u\cap Ox$, $S\neq Z$, such that $\overline{SZ}\cap W_X^u=\{S,Z\}$.  
	\end{lem}
	
	\begin{proof}
	Assume first that $W_X^{u+}$ intersects $Ox$ at points other than $Z$. Let $i\in\mathbb{N}$ be such that $[Z^{2i+2},Z^{2i+4}]^u$ intersects $Ox$. Then $[Z^{2i},Z^{2i+2}]^u$ intersects the preimage of $Oy$, that is, the curve $y=a|x|-1$ in the first or fourth quadrant ($\mathcal{Q}_1$ or $\mathcal{Q}_4$); see Figure \ref{fig:polygon_D_case1}.
	
	We claim that only finitely many consecutive forward images of $[Z^{2i},Z^{2i+2}]^u$ under $L_{a,b}^2$ intersect the preimage of $Oy$. Equivalently, there is $k\in\mathbb{N}$ such that 
	\begin{equation*}
	[Z^{2i+2k},Z^{2i+2k+2}]^u\cap L_{a,b}^{-1}(Oy)=\emptyset.
	\end{equation*} 
Suppose, for contradiction, that this is not the case. Then for every $j\in\mathbb{N}$,
	\begin{equation*}
	[Z^{2i+2j},Z^{2i+2j+2}]^u\cap L_{a,b}^{-1}(Oy)\neq\emptyset.
	\end{equation*}	 
In particular, $[Z^{2i},Z^{2i+2}]^u$ contains a point whose forward orbit under $L_{a,b}$ lies in $\mathcal{Q}_1$ below the straight line $\overline{ZZ^{1-}}^u$. Corollary \ref{cor:stable_iterates_1st_quadrant} now implies that this point lies on $W_X^s$. This is a contradiction with the assumption that the fixed point $X$ has no homoclinic points. 

	Therefore, there exists $k\in\mathbb{N}$ such that 
	\begin{equation*}
	[Z^{2i+2j},Z^{2i+2j+2}]^u\cap L_{a,b}^{-1}(Oy)\neq\emptyset \qquad \text{for all } j\in\{0,1,\ldots,k-1\},
	\end{equation*}
while
	\begin{equation*}
	[Z^{2i+2k},Z^{2i+2k+2}]^u\cap L_{a,b}^{-1}(Oy)=\emptyset.
	\end{equation*}			  
	
	Thus, $k$ is the smallest positive integer for which the segment $[Z^{2i+2k},Z^{2i+2k+2}]^u$ no longer intersects $L_{a,b}^{-1}(Oy)$. Let $A$ be the rightmost intersection point of $[Z^{2i+2k},Z^{2i+2k+2}]^u$ with $Ox$. We claim that $A$ is the desired point $S$. To prove this, define the polygon $\mathcal{D}$ with the boundary
	\begin{equation*}
	\partial\mathcal{D}=[Z,A]^u\cup \overline{AZ}.
	\end{equation*}	  
We will show that all forward images of $[Z^{2i+2k},Z^{2i+2k+2}]^u$ under $L_{a,b}^2$ remain in $\Int\mathcal{D}$. First, notice that $[Z^{2i+2k+2},Z^{2i+2k+4}]^u$ is contained in $\Int\mathcal{D}$ since it does not intersect $\partial\mathcal{D}$. Indeed, it cannot intersect $[Z,A]^u$ since $W_X^u$ does not intersect itself. Moreover, it does not intersect $\overline{AZ}$ either, since otherwise $[Z^{2i+2k},Z^{2i+2k+2}]^u$ would intersect $\overline{A^{-2}Z^{-2}}\subset L_{a,b}^{-1}(Oy)$, which is impossible by the choice of $k$. It follows that $[Z^{2i+2k+2},Z^{2i+2k+4}]^u\subset\Int\mathcal{D}$.

	Repeating this argument inductively, we see that for all $n\in\mathbb{N}$,
	\begin{equation*}
	L_{a,b}^{2n}([Z^{2i+2k},Z^{2i+2k+2}]^u)\subset\Int\mathcal{D},
	\end{equation*}
since the same non-intersection argument applies at each step. Indeed, none of these forward images can intersect $[Z,A]^u$, and they cannot intersect $\overline{AZ}$ since they do not intersect $\overline{A^{-2}Z^{-2}}$, which lies outside of $\mathcal{D}$. Consequently, $W_X^{u+}$ does not intersect the segment $\overline{AZ}$ at points other than $A$ and $Z$. This proves that $A$ is the desired point $S$ in this case.
	
	We now assume that $W_X^{u+}$ does not intersect $Ox$ at points other than $Z$. In that case, $W_X^{u+}$ intersects $Oy$, and $W_X^{u-}$ intersects $Ox$ (recall that we assume infinitely many intersections of $W_X^u$ with $Ox$ and $Oy$). Let $B_1,B_2,B_3$ be the first three consecutive intersections of $W_X^{u-}$ with $Ox$, counting from $X$; that is, $B_1,B_2,B_3\in W_X^{u-}\cap Ox$ are points such that $[X,B_1)^{u}\cap Ox=\emptyset$, $[X,B_2)^u \cap Ox =\{B_1\}$ and $[X,B_3)^u \cap Ox =\{B_1,B_2\}$. We define the polygon $\mathcal{E}$ with the boundary $\partial\mathcal{E}=[B_1,B_3]^u\cup\overline{B_1B_3}$. We claim that 
	\begin{equation*}
	W_X^{u-}\setminus[X,B_3]^u \subset \Int\mathcal{E}. 
	\end{equation*}		
	
	Indeed, $W_X^{u-}\setminus[X,B_3]^u$ does not intersect $[B_1,B_3]^u\subset W_X^{u}$. Moreover, it does not intersect $\overline{B_1B_3}$ either, since it does not intersect $\overline{B_1^{-2}B_3^{-2}}$ which lies outside $\mathcal{E}$. This proves our claim. Therefore, $W_X^{u-}$ does not intersect $Ox$ at any points to the right of $B_2$. It follows that $B_2$ is the desired point $S$ in this case. This finishes the proof.      	
	\end{proof}
	
	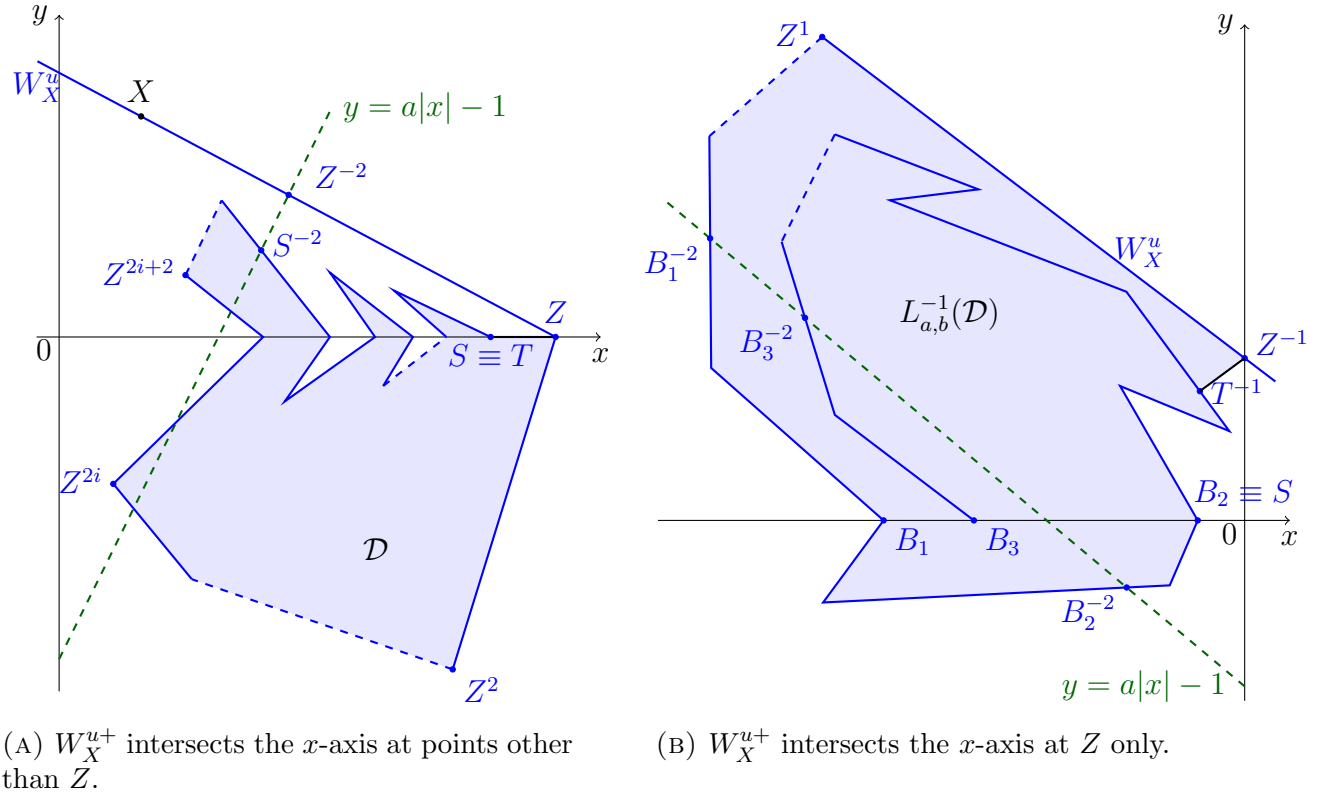
\begin{figure}[!ht]
	\begin{center}
	\begin{subfigure}[t]{0.475\textwidth}
	\begin{tikzpicture}[auto, scale=0.6, yscale=1.4285]
			\tikzstyle{nodec}=[draw,circle,fill=black,minimum size=2pt,
			inner sep=0pt, label distance=2mm]
			\tikzstyle{nodeh}=[draw,circle,fill=white,minimum size=4pt,
			inner sep=0pt]
			\tikzstyle{dot}=[circle,draw=none,fill=none,minimum size=0pt,inner sep=2pt, outer sep=-1pt]

			\draw[->] (-0.5,0)--(12,0) node [below]{$x$};
			\draw[->] (0,-5.5)--(0,5) node [left]{$y$};
			\node[label={[xshift=-0.2cm, yshift=-0.6cm]$0$}] at (0,0) {};
			
			\coordinate (e0) at (0,0);
			\coordinate (ex) at (1,0);
			\coordinate (ey) at (0,1);

			\coordinate (pr01) at (0,-5);
			\coordinate (pr02) at (6.04,3.56);
			
			\draw[green!40!black, thick, dashed] (pr01)--(pr02) node[right]{$y=a|x|-1$};
			
			\coordinate (a01) at (-0.48,4.28);
			\coordinate (a02) at (11,0);
			\coordinate (a03) at (8.72,-5.16);
			\coordinate (a04) at (2.94,-3.76);
			\coordinate (a05) at (1.2,-2.28);
			\coordinate (a06) at (4.52,0);
			\coordinate (a07) at (2.8,0.96);
			\coordinate (a08) at (3.6,2.12);
			\coordinate (a09) at (6,0);
			\coordinate (a10) at (5,-1);
			\coordinate (a11) at (7,0);
			\coordinate (a12) at (6,1);
			\coordinate (a13) at (7.84,0);
			\coordinate (a14) at (7.18,-0.76);
			\coordinate (a15) at (8.58,0);
			\coordinate (a16) at (7.42,0.72);
			\coordinate (a17) at (9.56,0);
			\coordinate (a18) at (8.18,-1.01);
			
			\draw[blue, thick] (a01) node[below]{$W_X^u$}--(a02)--(a03);
			\draw[blue, thick, dashed] (a03)--(a04);
			\draw[blue, thick] (a04)--(a05)--(a06)--(a07);
			\draw[blue, thick, dashed] (a07)--(a08);
			\draw[blue, thick] (a08)--(a09)--(a10)--(a11)--(a12)--(a13)--(a14);
			\draw[blue, thick, dashed] (a14)--(a15);
			\draw[blue, thick] (a15)--(a16)--(a17);
			
			\draw[thick] (a17)--(a02);
			
			\coordinate (x) at ($(a01)!0.2!(a02)$);
			
			\node[nodec, label={[above]$X$}] at (x) {};
			
			\node[nodec, blue, label={[above right, blue, xshift=2mm, yshift=-1mm]$Z^{-2}$}] at (intersection of a01--a02 and pr01--pr02) {};
			\node[nodec, blue, label={[above, blue]$Z$}] at (a02) {};
			\node[nodec, blue, label={[below right, blue]$Z^2$}] at (a03) {};
			\node[nodec, blue, label={[left, blue]$Z^{2i}$}] at (a05) {};
			\node[nodec, blue, label={[left, blue]$Z^{2i+2}$}] at (a07) {};
			\node[nodec, blue, label={[below, blue]$S\equiv T$}] at (a17) {};
			\node[nodec, blue, label={[right, blue, yshift=0.5mm]$S^{-2}$}] at (intersection of a08--a09 and pr01--pr02) {};
			
			\node[fill=blue!10!white] (dd) at (7.02,-3.3) {$\mathcal{D}$};
			
			\begin{pgfonlayer}{bg}
				\fill[blue!10!white] (a02.center)--(a03.center)--(a04.center)--(a05.center)--(a06.center)--(a07.center)--(a08.center)--(a09.center)--(a10.center)--(a11.center)--(a12.center)--(a13.center)--(a14.center)--(a15.center)--(a16.center)--(a17.center)--cycle;
			\end{pgfonlayer}
			\end{tikzpicture}
		\caption{$W_X^{u+}$ intersects the $x$-axis at points other than $Z$. \label{fig:polygon_D_case1}}
	\end{subfigure}
	\hfill
	\begin{subfigure}[t]{0.475\textwidth}
	\begin{tikzpicture}[auto, scale=0.6]
			\tikzstyle{nodec}=[draw,circle,fill=black,minimum size=2pt,
			inner sep=0pt, label distance=2mm]
			\tikzstyle{nodeh}=[draw,circle,fill=white,minimum size=4pt,
			inner sep=0pt]
			\tikzstyle{dot}=[circle,draw=none,fill=none,minimum size=0pt,inner sep=2pt, outer sep=-1pt]

			\draw[->] (-13,0)--(1,0) node [below]{$x$};
			\draw[->] (0,-4)--(0,11) node [left]{$y$};
			\node[label={[xshift=-0.2cm, yshift=-0.6cm]$0$}] at (0,0) {};
			
			\coordinate (e0) at (0,0);
			\coordinate (ex) at (1,0);
			\coordinate (ey) at (0,1);

			\coordinate (pr01) at (0,-3.68);
			\coordinate (pr02) at (-12.9,7.14);
			
			\draw[green!40!black, thick, dashed] (pr01)node[left, xshift=-1mm]{$y=a|x|-1$}--(pr02);
			
			\coordinate (a01) at (0.68,3.08);
			\coordinate (a02) at (-9.36,10.72);
			\coordinate (a03) at (-11.86,8.52);
			\coordinate (a04) at (-11.82,3.38);
			\coordinate (a05) at (-8,0);
			\coordinate (a06) at (-9.34,-1.82);
			\coordinate (a07) at (-1.66,-1.44);
			\coordinate (a08) at (-1.04,0);
			\coordinate (a09) at (-2.76,2.98);
			\coordinate (a10) at (-0.34,1.98);
			\coordinate (a11) at (-2.62,5.07);
			\coordinate (a12) at (-7.86,7.1);
			\coordinate (a13) at (-5.9,7.34);
			\coordinate (a14) at (-9.08,8.56);
			\coordinate (a15) at (-10.26,6.18);
			\coordinate (a16) at (-9.08,2.34);
			\coordinate (a17) at (-6,0);
			
			\draw[blue, thick] (a01)--(a02) node[pos=0.3, above]{$W_X^u$};
			\draw[blue, thick, dashed] (a02)--(a03);
			\draw[blue, thick] (a03)--(a04)--(a05)--(a06)--(a07)--(a08)--(a09)--(a10)--(a11)--(a12)--(a13)--(a14);
			\draw[blue, thick, dashed] (a14)--(a15);
			\draw[blue, thick] (a15)--(a16)--(a17);

			\coordinate (zm1) at (intersection of a01--a02 and e0--ey);
			
			\node[nodec, blue, label={[above right, blue, yshift=-1.5mm]$Z^{-1}$}] at (zm1) {};
			\node[nodec, blue, label={[left, blue]$Z^1$}] at (a02) {};
			\node[nodec, blue, label={[below right, blue]$B_1$}] at (a05) {};
			
			\node[nodec, blue, label={[above right, blue, xshift=-2mm]$B_2\equiv S$}] at (a08) {};
			
			\node[nodec, blue, label={[below right, blue]$B_3$}] at (a17) {};
			
			\node[nodec, blue, label={[below left, blue]$B_1^{-2}$}] at (intersection of a03--a04 and pr01--pr02) {};
			\node[nodec, blue, label={[below left, blue]$B_2^{-2}$}] at (intersection of a06--a07 and pr01--pr02) {};
			\node[nodec, blue, label={[below left, blue]$B_3^{-2}$}] at (intersection of a15--a16 and pr01--pr02) {};
			
			\coordinate (t) at ($(a10)!(zm1)!(a11)$);
			
			\node[nodec, blue, label={[right, blue, yshift=-0.5mm]$T^{-1}$}] at (t) {};
			
			\draw[thick] (zm1)--(t);
			
			\node[fill=blue!10!white] (ddm1) at (-6.54,4.56) {$L_{a,b}^{-1}(\mathcal{D})$};
			
			\begin{pgfonlayer}{bg}
				\fill[blue!10!white] (zm1.center)--(a02.center)--(a03.center)--(a04.center)--(a05.center)--(a06.center)--(a07.center)--(a08.center)--(a09.center)--(a10.center)--(t.center)--cycle;
			\end{pgfonlayer}
			\end{tikzpicture}
		\caption{$W_X^{u+}$ intersects the $x$-axis at $Z$ only. \label{fig:polygon_D_case2}}
	\end{subfigure}
	\end{center}
	\caption{Construction of the point $T$ and the polygon $\mathcal{D}$. \label{fig:polygon_D_construction}}
	\end{figure} 
	
	The proof of Lemma \ref{lem:point_S} shows the existence of a suitable polygon bounded by a segment of $W_X^u$ and a straight line segment. Since this polygon will be used repeatedly below, we now give a formal definition of it. We do so by introducing a point $T \in W_X^u$, constructed as follows.
	
	If $W_X^{u+}$ intersects $Ox$ at points other than $Z$, let $T := S$, where $S$ is the point as in Lemma \ref{lem:point_S}; see Figure \ref{fig:polygon_D_case1}.

	If $W_X^{u+}$ does not intersect $Ox$ at points other than $Z$, then the point $S$ lies on $W_X^{u-}$. Just as in the proof of Lemma \ref{lem:point_S}, let $B_1,B_2,B_3$ be the first three consecutive intersections of $W_X^{u-}$ with $Ox$, in the order along $W_X^{u-}$ starting from $X$; see Figure \ref{fig:polygon_D_case2}. From that same proof, we know that $S=B_2$. Moreover, $W_X^{u-}\setminus[X,B_3]^u$ is contained in a region bounded by $[B_1,B_3]^u\cup\overline{B_1B_3}$.
	
	Now, let $T^{-1}$ be the point on $[S,B_3]^u$ which is the closest (in the Euclidean metric) to $Z^{-1}$. If there are several such points (note that there can be at most finitely many of them), we choose $T^{-1}$ such that $[S,T^{-1}]^u$ contains all the others. In this case, we define $T$ as $T:=L_{a,b}(T^{-1})$. 
	
	In both cases, we define $\mathcal{D}$ to be the polygon with the boundary
	\begin{equation} \label{eq:polygon_D_boundary}
	\partial\mathcal{D}=[Z,T]^u\cup\overline{ZT}.
	\end{equation}
	
	Note that in both cases, $T$ lies on $W_X^{u+}$, satisfies $W_X^u\cap\overline{ZT}=\{Z,T\}$, and has the property that $W_X^{u+}\setminus[X,T]^u\subset\Int\mathcal{D}$. Geometrically, $\mathcal{D}$ is a region bounded by a portion of $W_X^u$ and the segment $\overline{ZT}$, which contains the forward branch $W_X^{u+}\setminus[X,T]^u$. 
	
	\begin{cor} \label{cor:polygon_D}
	Let $\mathcal{D}$ be the polygon whose boundary is given by (\ref{eq:polygon_D_boundary}). Then $\mathcal{D}$ is $L_{a,b}^2$-invariant, $L_{a,b}^2(\mathcal{D})\subset\mathcal{D}$.
	\end{cor}
	
	\begin{proof}
	By construction of $\mathcal{D}$ and the proof of Lemma \ref{lem:point_S}, it follows that $L_{a,b}^2(\partial\mathcal{D}\cap W_X^u)\subset\mathcal{D}$.
	
	It remains to consider the image of the segment $\overline{ZT}$ under $L_{a,b}^2$. We know that $L_{a,b}^2(\overline{ZT})$ intersects $W_X^u$ at its endpoints only, namely at $Z^2$ and $T^2$. Moreover, $L_{a,b}^2(\overline{ZT})$ does not intersect $\overline{ZT}$, since $\overline{ZT}$ does not intersect its preimage $L_{a,b}^{-2}(\overline{ZT})=\overline{Z^{-2}T^{-2}}$, which lies outside $\mathcal{D}$. It follows that $L_{a,b}^2(\overline{ZT})\subset\mathcal{D}$. 
		
	Therefore, $L_{a,b}^2(\partial\mathcal{D})\subset\mathcal{D}$, which implies the claim and completes the proof. 
	\end{proof}
	
	In addition, the previous proof implies that the boundary $\partial\mathcal{D}$ maps to the interior $\Int\mathcal{D}$ under finitely many iterates of $L_{a,b}^{2}$. More precisely, there exists $k \in \mathbb{N}$ such that $L_{a,b}^{2k}(\partial\mathcal{D}) \subset \Int\mathcal{D}$. Therefore, for the same $k$, we also have $L_{a,b}^{2k}(\mathcal{D}) \subset \Int\mathcal{D}$. This motivates the following notion, which formalizes this type of dynamical behavior.
	
	\begin{df} \label{dfn:eventually_trapping_region}
	Let $f \colon \mathcal{X} \rightarrow \mathcal{X}$ be a continuous map on a topological space $\mathcal{X}$. A set $\mathcal{A} \subset \mathcal{X}$ is called an \emph{eventually trapping region} for $f$ if there exists $n \in \mathbb{N}$ such that
	\begin{equation*}
	f^n(\Cl\mathcal{A}) \subset \Int\mathcal{A}.
	\end{equation*}
	\end{df}
	
	The following result is now an immediate consequence of Definition \ref{dfn:eventually_trapping_region} and the paragraph preceding it.
	 
	 \begin{cor} \label{cor:D_eventually_trapping}
	 Polygon $\mathcal{D}$ enclosed by the boundary given in (\ref{eq:polygon_D_boundary}) is an eventually trapping region for $L_{a,b}^2$.
	 \end{cor} 		 
		
	Now, we will show that our geometric construction of $\mathcal{D}$ implies that every compact subarc of $W_X^u$ is isolated from the rest of $W_X^u$.
	
	Let $M, N \in W^u_X$ be two distinct points and let $0 < \varepsilon < \frac{1}{2}\dist(M, N)$ (where $\dist$ is the Euclidean metric in $\mathbb{R}^2$). Recall that $(M, N)^{u} = [M, N]^{u} \setminus \{ M, N \}$. We say that a point $M' \in (M, N)^{u}$ is \emph{the first point along $(M,N)^u$ at distance $\varepsilon$ from $M$} if $\dist(M, M') = \varepsilon$ and for every $Q \in (M, N)^{u}$ such that $\dist(M, Q) = \varepsilon$, we have $(M, M')^{u} \subseteq (M, Q)^{u}$. 
	
	Let $N' \in (M, N)^{u}$ be the first point along $(M,N)^u$ at distance $\varepsilon$ from $N$. We define
	
	\begin{equation*}
	B_{\varepsilon}((M, N)^{u}) := \bigcup_{Q \in [M', N']^{u}}B_{\varepsilon}(Q).	
	\end{equation*} 
	
	\begin{cor}\label{cor:eps}
		For every $Q \in W^u_X$, $Q\neq X$, there exists $\varepsilon > 0$ such that $B_\varepsilon((X, Q)^{u}) \cap W^u_X = (X, Q)^{u}$.
	\end{cor}
	\begin{proof}
	We know from Lemma \ref{lem:point_S} that $W_X^u\setminus(Z,Z^1)^u$ is contained in $\mathcal{D}\cup L_{a,b}(\mathcal{D})$, and the union of these two polygons does not intersect the segment $[X,Z^{-1}]^u$. Therefore, there exists $\delta>0$ such that $B_{\delta}((X,Z^{-1})^u)\cap W_X^u=(X,Z^{-1})^u$.
	
	Now, let $Q\in W_X^u\setminus\{X\}$ be arbitrary, but fixed. There exists $k\in\mathbb{N}_0$ such that $Q^{-k}\in(X,Z^{-1})^u$, that is, $(X,Q^{-k})^u\subset(X,Z^{-1})^u$. Therefore, $L_{a,b}^k(B_{\delta}((X,Z^{-1})^u))$ is an open neighborhood of $(X,Q)^u$ such that $(X,Q)^u\cap L_{a,b}^k(B_{\delta}((X,Z^{-1})^u))=(X,Q)^u$. Within that open neighborhood, we can find $\varepsilon>0$ such that $B_{\varepsilon}((X,Q)^u)$ is contained in it. This $\varepsilon$ satisfies the claim, which completes the proof.
	\end{proof}
	
	Recall that $\ell$ denotes the set of all accumulation points of $W_X^u$ which do not lie on $W_X^u$, that is, $\ell = \Cl W_X^u \setminus W_X^u$. Since Corollary \ref{cor:eps} implies that $W_X^u$ locally is an arc, we see that $\ell$ is also the $\omega$-limit set of $W_X^u$, that is, $\ell = \omega(W_X^u,L_{a,b})$.
	
	Furthermore, let $\mathcal{D}$ be the polygon enclosed by the boundary given in (\ref{eq:polygon_D_boundary}). Notice that the accumulation set $\ell$ of $W_X^u$ can be represented as a union
	\begin{equation*}
	\ell=\ell_L\cup\ell_R,
	\end{equation*}
	where $\ell_R=\ell\cap\mathcal{D}$ and $\ell_L=\ell\cap L_{a,b}(\mathcal{D})$, corresponding to the left and right parts of the accumulation set. Observe that both $\ell_R$ and $\ell_L$ are $L_{a,b}^2$-invariant and $L_{a,b}^{-2}$-invariant. In fact, $\ell_R = \omega(W_X^{u+}, L_{a,b}^2)$ and $\ell_L = \omega(W_X^{u-},L_{a,b}^2)$.
	
	\begin{figure}[!ht]
	\begin{center}
	\begin{tikzpicture}[auto, scale=1.25]
			\tikzstyle{nodec}=[draw,circle,fill=black,minimum size=2pt,
			inner sep=0pt, label distance=2mm]
			\tikzstyle{nodeh}=[draw,circle,fill=white,minimum size=4pt,
			inner sep=0pt]
			\tikzstyle{dot}=[circle,draw=none,fill=none,minimum size=0pt,inner sep=2pt, outer sep=-1pt]

			\draw[->] (-1.5,0)--(6,0) node [below]{$x$};
			\draw[->] (0,-4.5)--(0,2.5) node [left]{$y$};
			\node[label={[xshift=-0.2cm, yshift=-0.6cm]$0$}] at (0,0) {};
			
			\coordinate (e0) at (0,0);
			\coordinate (ex) at (1,0);
			\coordinate (ey) at (0,1);

			\coordinate (a01) at (-0.78,1.92);
			\coordinate (a02) at (5,0);
			\coordinate (a03) at (5.33,-2.96);
			\coordinate (a04) at (-0.97,-4.34);
			\coordinate (a05) at (0.6,-1.65);
			\coordinate (a06) at (-0.67,-1.32);
			\coordinate (a07) at (1.26,0.45);
			\coordinate (a08) at (2.36,0);
			\coordinate (a09) at (1.16,-0.7);
			\coordinate (a10) at (3,0);
			\coordinate (a11) at (2.21,0.57);
			\coordinate (a12) at (4.27,0);
			\coordinate (a13) at (3.09,-1.82);
			\coordinate (a14) at (4.04,-1.44);
			\coordinate (a15) at (3.9,-2.63);
			\coordinate (a16) at (1.55,-2.68);
			\coordinate (a17) at (2.19,-3.27);
			\coordinate (a18) at (0.7,-2.67);
			\coordinate (a19) at (0.84,-1.85);
			
			\node[nodec, color=blue, label={[above right, color=blue]$Z$}] at (a02) {};
			\node[nodec, color=blue, label={[right, color=blue]$Z^2$}] at (a03) {};
			\node[nodec, color=blue, label={[left, color=blue]$Z^4$}] at (a04) {};
			\node[nodec, color=blue, label={[left, color=blue]$Z^6$}] at (a06) {};
			\node[nodec, color=blue, label={[below left, xshift=-2mm, color=blue]$T$}] at (a12) {};
			\node[nodec, color=blue, label={[above left, color=blue]$T^2$}] at (a15) {};
			\node[nodec, color=blue, label={[above right, color=blue]$T^4$}] at (a18) {};
			
			\draw[thick] (a03)--(a15);
			\draw[thick] (a04)--(a18);
			
			\draw[blue, thick] (a02)--(a01) node[above]{$W_X^u$};
			\draw[blue, thick] (a02)--(a03)--(a04)--(a05)--(a06);
			\draw[blue, thick, dashed] (a06)--(a07);
			\draw[blue, thick] (a07)--(a08)--(a09)--(a10)--(a11)--(a12)--(a13)--(a14)--(a15)--(a16)--(a17)--(a18);
			\draw[blue, thick, dashed] (a18)--(a19);
			
			\node[fill=green!10!white] (kk) at (4.5,-1.26) {$\mathcal{K}$};
			\node[fill=red!10!white] (l2k) at (0.92,-3.42) {$L_{a,b}^2(\mathcal{K})$};
			
			\begin{pgfonlayer}{bg}
				\fill[green!10!white] (a02.center)--(a03.center)--(a15.center)--(a14.center)--(a13.center)--(a12.center)--cycle;
				
				\fill[red!10!white] (a03.center)--(a04.center)--(a18.center)--(a17.center)--(a16.center)--(a15.center)--cycle;
				
				\fill[pattern=north east lines] (a02.center)--(a03.center)--(a15.center)--(a14.center)--(a13.center)--(a12.center)--cycle;
				
				\fill[pattern=dots] (a03.center)--(a04.center)--(a18.center)--(a17.center)--(a16.center)--(a15.center)--cycle;
			\end{pgfonlayer}

			\end{tikzpicture}
	\end{center}
	\caption{Polygon $\mathcal{K}$ and its image under $L_{a,b}^2$ (Lemma \ref{lem:P_intersection_D'}). \label{fig:entropy_ell_K}}
	\end{figure}
	
	Recall that for a point $A \in \mathbb{R}^2$ and a subset $\mathcal{S} \subseteq \mathbb{R}^2$, the Euclidean distance from $A$ to $\mathcal{S}$ is $\dist(A,\mathcal{S})=\inf_{B \in \mathcal{S}}\dist(A,B)$.
	
	\begin{lem} \label{lem:P_intersection_D'}
	Let $\mathcal{D}$ be the polygon enclosed by the boundary given in (\ref{eq:polygon_D_boundary}). Then
	\begin{equation*}
	\ell_R=\bigcap_{k=0}^{\infty}L_{a,b}^{2k}(\mathcal{D}).
	\end{equation*}
	\end{lem}
	
	\begin{proof}
	Since $\ell_R\subseteq\mathcal{D}$, $\mathcal{D}$ is $L_{a,b}^2$-invariant (Corollary \ref{cor:polygon_D}) and $\ell_R$ is both $L_{a,b}^2$-invariant and $L_{a,b}^{-2}$-invariant, it directly follows that $\ell_R\subseteq\bigcap_{k=0}^{\infty}L_{a,b}^{2k}(\mathcal{D})$.
	
	It remains to show that $\bigcap_{k=0}^{\infty}L_{a,b}^{2k}(\mathcal{D})\subseteq\ell_R$. To prove that, let $\mathcal{K}$ be the polygon with the boundary
	\begin{equation} \label{eq:polygon_K_boundary}
	\partial\mathcal{K}=\overline{TZ}\cup[Z,Z^2]^{u}\cup[T,T^2]^{u}\cup L_{a,b}^2(\overline{TZ}).
	\end{equation}
	We see that $\Int L_{a,b}^{2j}(\mathcal{K})\cap\Int L_{a,b}^{2k}(\mathcal{K})=\emptyset$ for all $j,k\in\mathbb{N}_0$, $j\neq k$ (see Figure \ref{fig:entropy_ell_K}).
	
	Notice that
	\begin{equation*}
	L_{a,b}^{2}(\mathcal{D})=\mathcal{D} \setminus \bigl( \Int\mathcal{K} \cup \overline{ZT} \cup \overline{ZZ^2}^u \bigr),
	\end{equation*}
and similarly, for every $k\in\mathbb{N}$,
	\begin{equation} \label{eq:polygon_D_iterates}
	L_{a,b}^{2k}(\mathcal{D})=\mathcal{D}\setminus \left( \bigcup_{j=0}^{k-1}\Int L_{a,b}^{2j}(\mathcal{K}) \cup \bigcup_{j=0}^{k-1} L_{a,b}^{2j}(\overline{ZT}) \cup \bigcup_{j=0}^{k-1} L_{a,b}^{2j}(\overline{ZZ^2}^u) \right).
	\end{equation}
	
	Now, take a point $A \in \bigcap_{k=0}^{\infty}L_{a,b}^{2k}(\mathcal{D})$. Since $|\det DL_{a,b}|=b<1$, we see that $L_{a,b}$ is dissipative. In particular, the area of $L_{a,b}^{2k}(\mathcal{D})$ tends to zero as $k \rightarrow \infty$. The polygons $L_{a,b}^{2k}(\mathcal{D})$ form a nested sequence of compact sets containing $A$. Hence, for any $\varepsilon>0$, the open ball $B_{\varepsilon}(A)$ has a strictly larger area than $L_{a,b}^{2k}(\mathcal{D})$ for all sufficiently large $k$, and therefore cannot be contained in $L_{a,b}^{2k}(\mathcal{D})$. It follows that the Euclidean distance from $A$ to the boundary $\partial L_{a,b}^{2k}(\mathcal{D})=L_{a,b}^{2k}(\partial\mathcal{D})$ also tends to zero as $k\rightarrow\infty$.  Consequently, $A$ belongs to the $\omega$-limit set of the boundary $\partial\mathcal{D}$, that is, 
	\begin{equation} \label{eq:A_omega_limit_boundary_D}
	A \in \omega(\partial\mathcal{D}, L_{a,b}^2)=\omega([Z,T]^u\cup\overline{ZT}, L_{a,b}^2).
	\end{equation}
	
	On the other hand, for every $k \in \mathbb{N}$, from (\ref{eq:polygon_D_iterates}) it follows that $\mathcal{D} \setminus L_{a,b}^{2k}(\mathcal{D})$ contains $L_{a,b}^{2j}(\overline{ZT})$ for all $j = 0,1,\ldots,k-1$. At the same time, we know that $L_{a,b}^{2k}(\mathcal{D})$ contains $[Z^{2j},Z^{2j+2}]^u$ for all $j\geqslant k$.
	
	As a consequence, for every $i \in \mathbb{N}_0$, there exists $j \in \mathbb{N}_0$, $j > i$, such that the Euclidean distance from $A$ to $[Z^{2j},Z^{2j+2}]^u$ is less than the one from $A$ to $L_{a,b}^{2i}(\overline{ZT})$. Combining this with (\ref{eq:A_omega_limit_boundary_D}), we obtain $A \in \omega([Z,T]^u,L_{a,b}^{2})$, from which it follows that $A \in \ell$. This completes the proof.   
	\end{proof}
	
	\begin{cor} \label{cor:ell_intersect_D}
	Let $\mathcal{D}$ be as in Lemma \ref{lem:P_intersection_D'}. Then
	\begin{equation*}
	\ell_L=\bigcap_{k=0}^{\infty} L_{a,b}^{2k+1}(\mathcal{D}),\qquad \ell=\bigcap_{k=0}^{\infty} L_{a,b}^{k}\bigl(\mathcal{D} \cup L_{a,b}(\mathcal{D})\bigr).
	\end{equation*}
	\end{cor}
	
	From Corollary \ref{cor:polygon_D} and the Brouwer Fixed Point Theorem, we know that $\mathcal{D}\cup L_{a,b}(\mathcal{D})$ contains the periodic orbit $\{P,P'\}$. Combined with Corollary \ref{cor:ell_intersect_D}, this allows us to conclude the following result.
	
	\begin{cor}\label{cor:P_in_Ell}
	The periodic orbit $\{P,P'\}$ is contained in $\ell$.
	\end{cor}
	
	Corollary \ref{cor:ell_intersect_D} implies that $\ell$ arises as a compact, maximal invariant set in an eventually trapping region $\mathcal{D}$. Heuristically, $\ell$ plays the role of an attractor for $L_{a,b}$, while its connected components $\ell_L$ and $\ell_R$ play the role of attractors for $L_{a,b}^2$.

	\subsection{Main result} \label{subsec:main_result} 
	
	Recall that $\mathcal{R}\subset\mathfrak{R}$ is the set of parameter pairs such that for every $n\in\mathbb{N}_0$, point $Z^{2n-1}$ belongs to the left and $Z^{2n}$ to the right half-plane. We know from \cite{misiurewicz2024zero} that $h_{top}(L_{a,b})=0$ for all $(a,b) \in \mathcal{R}$. We will expand this result to parameter pairs $(a,b) \in \mathfrak{R}\setminus\mathcal{R}$, that is, parameters such that the unstable manifold $W_X^u$ intersects the coordinate axes at points other than $Z$ and $Z^{-1}$.
	
	Let $W_Y^{u+}$ denote the lower connected component of the fixed point $Y$ in the third quadrant $\mathcal{Q}_3$ (it is a ray starting at $Y$ and going down in $\mathcal{Q}_3$). In the extended plane, we define the set
	\begin{equation*}
	\mathcal{N}:=W_X^u\cup\ell\cup W_X^{s+}\cup W_Y^{u+}\cup\{\infty\}.
	\end{equation*}
By construction, $\mathcal{N}$ is $L_{a,b}$-invariant and $L_{a,b}^{-1}$-invariant, compact, connected in the extended plane, and does not separate either the extended plane or the plane.
	
	Let $U'$ be the complement of $\mathcal{N}$ in the extended plane. Then $U'$ is $L_{a,b}$-invariant and $L_{a,b}^{-1}$-invariant by construction and, due to Corollary \ref{cor:P_in_Ell}, does not contain any fixed points of $L_{a,b}^2$.
	
		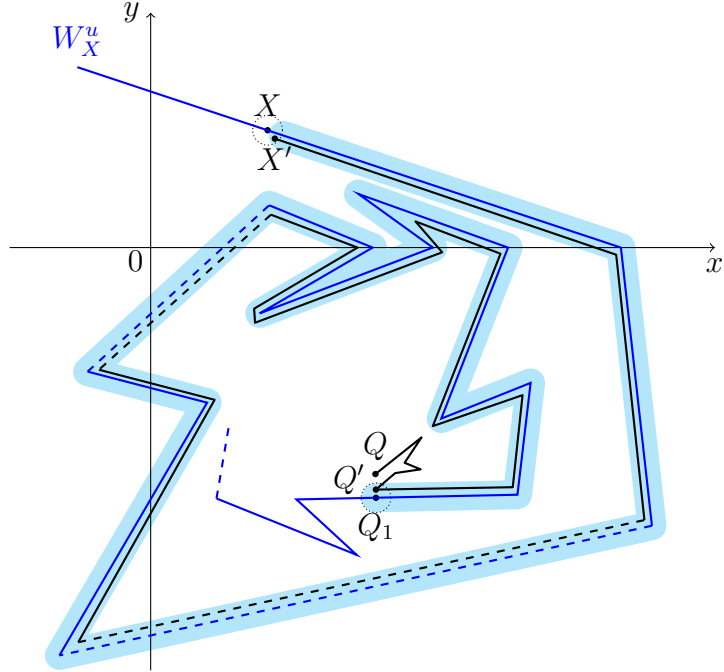
\begin{figure}[!ht]
	\begin{center}
	\begin{tikzpicture}[auto, scale=1.25]
			\tikzstyle{nodec}=[draw,circle,fill=black,minimum size=2pt,
			inner sep=0pt, label distance=2mm]
			\tikzstyle{nodeh}=[draw,circle,fill=white,minimum size=4pt,
			inner sep=0pt]
			\tikzstyle{dot}=[circle,draw=none,fill=none,minimum size=0pt,inner sep=2pt, outer sep=-1pt]

			\draw[->] (-1.5,0)--(6,0) node [below]{$x$};
			\draw[->] (0,-4.5)--(0,2.5) node [left]{$y$};
			\node[label={[xshift=-0.2cm, yshift=-0.6cm]$0$}] at (0,0) {};
			
			\coordinate (e0) at (0,0);
			\coordinate (ex) at (1,0);
			\coordinate (ey) at (0,1);
			
			\coordinate (q) at (2.39,-2.41);

			\coordinate (a01) at (-0.78,1.92);
			\coordinate (a02) at (5,0);
			\coordinate (a03) at (5.33,-2.96);
			\coordinate (a04) at (-0.97,-4.34);
			\coordinate (a05) at (0.6,-1.65);
			\coordinate (a06) at (-0.67,-1.32);
			\coordinate (a07) at (1.26,0.45);
			\coordinate (a08) at (2.36,0);
			\coordinate (a09) at (1.16,-0.7);
			\coordinate (a10) at (3,0);
			\coordinate (a11) at (2.21,0.57);
			\coordinate (a12) at (3.8,0); %(4.17,0);
			\coordinate (a13) at (3.09,-1.82);
			\coordinate (a14) at (4.04,-1.44);
			\coordinate (a15) at (3.9,-2.63);
			\coordinate (a16) at (1.55,-2.68);
			\coordinate (a17) at (2.19,-3.27);
			\coordinate (a18) at (0.7,-2.67);
			\coordinate (a19) at (0.84,-1.85);
			
			\coordinate (q1) at ($(a15)!(q)!(a16)$);
			\node[nodec, label={[above]$Q$}] at (q) {};
			\node[nodec, label={[below, yshift=-1mm]$Q_1$}] at (q1) {};
			
			\coordinate (x) at ($(a01)!0.35!(a02)$);
			\node[nodec, label={[above]$X$}] at (x) {};
			
			\draw[densely dotted] (q1) circle (4.5pt);
			\draw[densely dotted] (x) circle (4.5pt);
			
			\coordinate (qp) at ($(q1)!2.5pt!(q)$);
			\node[nodec, label={[left, yshift=1mm]$Q'$}] at (qp) {};
			
			\coordinate (b01) at (3.85,-2.55);
			\coordinate (b02) at (3.9544,-1.57);
			\coordinate (b03) at (3,-1.9);
			\coordinate (b04) at (3.72,-0.06538);
			\coordinate (b05) at (2.81532,0.2785);
			\coordinate (b06) at (3.1,-0.05);
			\coordinate (b07) at (1.11,-0.8);
			\coordinate (b071) at (1.1,-0.65);
			\coordinate (b08) at (2.2,0);
			\coordinate (b09) at (1.27863,0.35);
			\coordinate (b10) at (-0.55,-1.3);
			\coordinate (b11) at (0.68,-1.62);
			\coordinate (b12) at (-0.77,-4.2);
			\coordinate (b13) at (5.25,-2.9);
			\coordinate (b14) at (4.95,-0.07612);
			\coordinate (b15) at (1.32,1.15967);
			
			\coordinate (start) at ($(x)!5pt!(a02)$);
			
			\node[nodec, label={[below]$X'$}] at (b15) {};
			
			\draw[blue, thick] (a02)--(a01) node[above]{$W_X^u$};
			\draw[blue, thick] (a02)--(a03);
			\draw[blue, thick, dashed] (a03)--(a04);
			\draw[blue, thick] (a04)--(a05)--(a06);
			\draw[blue, thick, dashed] (a06)--(a07);
			\draw[blue, thick] (a07)--(a08)--(a09)--(a10)--(a11)--(a12)--(a13)--(a14)--(a15)--(a16)--(a17)--(a18);
			\draw[blue, thick, dashed] (a18)--(a19);
			
			\coordinate (c1) at (2.88,-2.02);
			\coordinate (c2) at (2.7,-2.29);
			\coordinate (c3) at (2.87,-2.36);
			\coordinate (c4) at (2.6,-2.4);			
			
			\draw [black, thick, join=round] (q)--(c1)--(c2)--(c3)--(c4)--(qp)--(b01)--(b02)--(b03)--(b04)--(b05)--(b06)--(b07)--(b071)--(b08)--(b09);
			\draw [black, thick, dashed, join=round] (b09)--(b10);
			\draw [black, thick, join=round] (b10)--(b11)--(b12);
			\draw[black, thick, dashed] (b12)--(b13);
			\draw [black, thick, join=round] (b13)--(b14)--(b15);

		\begin{pgfonlayer}{bg}
	\draw [line width=11pt, cap=round, join=round, cyan, opacity=0.3] (start)--(a02)--(a03)--(a04)--(a05)--(a06)--(a07)--(a08)--(a09)--(a10)--(a11)--(a12)--(a13)--(a14)--(a15)--(q1);	
    	\end{pgfonlayer}
 	
			\end{tikzpicture}
	\end{center}
	\caption{Construction of a path from $Q$ to $X'$ (Lemma \ref{lem:U'_pathwise}). The path from $Q'$ to $X'$ is contained in an open neighborhood of $(X,Q_1)^{u}$ which contains only one arc component of $W_X^u$ (Corollary \ref{cor:eps}).\label{fig:U'_pathwise}} 
	\end{figure}
	
	\begin{lem} \label{lem:U'_pathwise}
	The set $U'$ is path-connected.
	\end{lem}
	
	\begin{proof}
	Let the polygon $\mathcal{D}$ and the point $T$ be as in (\ref{eq:polygon_D_boundary}). Notice that the complement of the set $\mathcal{N}\cup\mathcal{D}\cup L_{a,b}(\mathcal{D})$ in the extended plane is path-connected since it is open and connected. Therefore, to prove that $U'$ is path-connected, it suffices to show that the set $\mathcal{D}\setminus(W_X^u\cup\ell)$ is path-connected (and this will imply that the same holds for $L_{a,b}(\mathcal{D})\setminus(W_X^u\cup\ell)$). 
	
	In order to prove that $\mathcal{D}\setminus(W_X^u\cup\ell)$ is path-connected, we will show the following: for any point $Q$ lying in $\mathcal{D}\setminus(W_X^u\cup\ell)$, there is a path from $Q$ to a point lying in an $\varepsilon$-neighborhood of the fixed point $X$, for a specially chosen $\varepsilon>0$; see Figure \ref{fig:U'_pathwise}.
	
	Let $Q\in\mathcal{D}\setminus(W_X^u\cup\ell)$ be arbitrary but fixed. Since $Q\notin\ell$, Lemma \ref{lem:P_intersection_D'} implies that there exists $k\in\mathbb{N}$ such that $Q\notin L_{a,b}^{2k}(\mathcal{D})$. Let $k_0$ be the smallest such $k$. We then have $Q\in L_{a,b}^{2k_0-2}(\mathcal{D})$ and $Q\notin L_{a,b}^{2k_0}(\mathcal{D})$. From (\ref{eq:polygon_D_iterates}) follows that $Q \in L_{a,b}^{2k_0-2}(\mathcal{K})$, where $\mathcal{K}$ is the polygon given by (\ref{eq:polygon_K_boundary}).
	
	Let $Q_1 \in [T^{2k_0-2},T^{2k_0}]^u$ be arbitrary, but fixed. Since $L_{a,b}^{2k_0-2}(\mathcal{K})$ is a polygon, there exists a path $\mathbf{b}_1 \subset L_{a,b}^{2k_0-2}(\mathcal{K})$ joining $Q$ to $Q_1$, and that path can be chosen such that it does not intersect $\partial L_{a,b}^{2k_0-2}(\mathcal{K})$ except at $Q_1$ and possibly $Q$. In other words, $\mathbf{b}_1 \setminus \{Q,Q_1\} \subset \Int L_{a,b}^{2k_0-2}(\mathcal{K})$, so it follows that $\mathbf{b}_1 \cap \ell = \emptyset$ and $\mathbf{b}_1 \cap W_X^u = \{Q_1\}$.
	
	On the other hand, Corollary \ref{cor:eps} implies the existence of $\varepsilon>0$ such that 
	$$B_{\varepsilon}((X,T^{2k_0})^{u})\cap W_X^u=(X,T^{2k_0})^{u}.$$ In addition, the set $B_{\varepsilon}((X,T^{2k_0})^{u})\setminus[X,T^{2k_0}]^{u}$ consists of two connected components. Exactly one of these components intersects $\mathbf{b}_1$. We denote that connected component by $\mathcal{B}$. By Corollary \ref{cor:eps}, the neighborhood $B_{\varepsilon}((X,T^{2k_0})^{u})$ contains no points of $\ell$; hence $\mathcal{B}\cap\ell=\emptyset$.
	
	Let $Q'$ be any point lying on $\mathcal{B} \cap (\mathbf{b}_1 \setminus Q_1)$, and let $X'$ be any point lying in $(\mathcal{B}\cap B_{\varepsilon}(X))\setminus W_X^{s+}$. Notice that $X'$ can be chosen such that it lies outside of $\mathcal{D}$. Since $\mathcal{B}$ is open and connected, it is also path-connected. Therefore, there exists a path $\mathbf{b}_2 \subset\mathcal{B}$ from $Q'$ to $X'$. By construction, the path $\mathbf{b}_2$ lies in $U'$, which finishes the proof.   
	\end{proof}
	
	We conclude with the proof of our main theorem and its consequence. 
	
	\begin{proof}[Proof of Theorem \ref{thm:L2_non_wandering_set}]
	Let $U'$ and $\mathcal{N}$ be defined as above. Lemma \ref{lem:U'_pathwise} implies that $U'$ is connected and since its complement $\mathcal{N}$ in the extended plane is also connected, it follows that $U'$ is simply connected in the extended plane and in the plane.
	
	The Riemann mapping theorem (RMT) states that for every simply connected open subset $D$ of the plane which is not the whole plane, there exists a biholomorphic map of $D$ onto the unit disk. Therefore, by RMT we have that $U'$ is homeomorphic to the open unit disk and moreover to the plane. 
	
	Moreover, $L_{a,b}^2\restrict{U'} \colon U' \rightarrow U'$ is an orientation-preserving, fixed-point-free homeomorphism. Therefore, by Corollary \ref{cor:BPTT_corollary}, every point of $U'$ is a wandering point for $L_{a,b}^2$. 
	
	Furthermore, every point in $(W_X^u \setminus \{X\}) \cup (W_X^{s+} \setminus \{X\}) \cup (W_Y^{u+} \setminus \{Y\})$ is a wandering point for $L_{a,b}$ and hence a wandering point for $L_{a,b}^2$ as well. We thus conclude that every non-wandering point for $L_{a,b}^2$ is either $X$, $Y$, or a point in $\ell$, which completes the proof.
	\end{proof}
	
	\begin{proof}[Proof of Corollary \ref{cor:entropy_gothR}]
	Theorem \ref{thm:L2_non_wandering_set} implies that the non-wandering set of the restriction $L_{a,b}^2\restrict{\mathbb{R}^2\setminus\ell}$ consists of its fixed points only, that is, $\Omega(L_{a,b}^2\restrict{\mathbb{R}^2\setminus\ell}) = \{X,Y\}$. By (\ref{eq:entropy_non_wandering_restriction}), we have that $h_{top}(L_{a,b}^2\restrict{\mathbb{R}^2\setminus\ell})=0$. Since $h_{top}(L_{a,b}^2\restrict{\mathbb{R}^2\setminus\ell})=2h_{top}(L_{a,b}\restrict{\mathbb{R}^2\setminus\ell})$, the claim follows. 
	\end{proof}
	
	\section{Concluding Remarks} \label{sec:concluding_remarks}
	
	For parameters $(a,b)\in\mathcal{R}$, we know from \cite{misiurewicz2024zero} that the accumulation set $\ell$ is $\ell=\{P,P'\}$, and that $h_{top}(L_{a,b}\restrict{\mathbb{R}^2\setminus\ell})=h_{top}(L_{a,b})=0$. Therefore, we see that Corollary \ref{cor:entropy_gothR} also holds for parameter values in $\mathcal{R}$, and it generalizes Theorem 1.1 in \cite{misiurewicz2024zero} to a larger set of parameters.
	
	Guided by this, one might think that $\ell=\{P,P'\}$ could hold for all parameters $(a,b)\in\mathfrak{R}$, that is, that the topological entropy of $L_{a,b}$ is zero on that whole parameter set. However, in a private communication, Y.\ Ishii and D.\ Sands have pointed out to the author that for $(a,b) = (1.6,0.61)$, the Lozi map $L_{a,b}$ has saddle period-six points whose stable and unstable manifolds intersect transversely. In addition, they also remark that other counterexamples exist, such as parameter pairs in a small neighborhood of $(a,b) = (1.5,0.5)$. These findings indicate positive topological entropy of $L_{a,b}$ (see, e.g., \cite{burns1995geometric}), and moreover, they imply that the accumulation set $\ell$ can have a nontrivial structure for parameters in $\mathfrak{R}\setminus\mathcal{R}$. Therefore, a complete characterization of the zero entropy locus for the Lozi map remains unresolved, and determining it precisely continues to be an open problem that calls for further investigation.

	\section*{Acknowledgements}
	
	This work was supported in part by the Croatian Science Foundation grants MOBODL-2023-08-4960, IP-2022-10-9820 GLODS, UIP-2025-02-4309 ToGeCoP, and in part by the European Research Executive Agency grant 101183111-DSYREKI-HORIZON-MSCA-2023-SE-01.
	
	The author thanks S.\ \v{S}timac for reading the text and the constructive advice on its improvements, as well as J.\ Boro\'nski for useful discussions concerning certain proofs in this work. The author would also like to thank Y.\ Ishii and D.\ Sands for sharing their findings which provided a new perspective on the zero entropy locus of the Lozi family and on the potential nontrivial structure of $\ell$ as well. Finally, the author expresses his gratitude to the anonymous Reviewers whose constructive comments improved the exposition of the paper.

	\section*{Author declarations}
	
	\subsection*{Conflict of interest}
	
	The author has no conflicts to disclose.
	
	\subsection*{Author contributions}
	
	Kristijan Kilassa Kvaternik: Conceptualization (Lead); Formal analysis (Lead); Investigation (Lead); Methodology (Lead); Writing - original draft (Lead).

\end{document}